\documentclass[10pt]{amsart}
\usepackage{fancyhdr,xcolor,tikz}
\usepackage{mathtools,amssymb,thmtools,array,hyperref}
\usepackage{booktabs}
\usepackage{cleveref}
\definecolor{mylinkcolor}{rgb}{0.5,0.0,0.0}
\definecolor{myurlcolor}{rgb}{0.0,0.0,0.75}
\definecolor{mydeficolor}{rgb}{0.0,0.5,0.0}
\hypersetup{colorlinks=true,urlcolor=myurlcolor,citecolor=myurlcolor,linkcolor=mylinkcolor,linktoc=page,breaklinks=true}

\title[Abelian surfaces of small conductor from genus 3 double covers]{Abelian surfaces of small conductor\\ from genus 3 double covers}
\author[R. van Bommel]{Raymond van Bommel}
\address{R. van Bommel, University of Bristol, r.vanbommel@bristol.ac.uk}
\author[C. Maistret]{C\'eline Maistret$^{\ast}$}
\address{C. Maistret, University of Bristol, celine.maistret@bristol.ac.uk}
\author[J. Shi]{Jia Shi}
\address{J. Shi, Massachusetts Institute of Technology, janeshi@mit.edu  }
\author[A.V. Sutherland]{Andrew V.\ Sutherland}
\address{A.V. Sutherland, Massachusetts Institute of Technology, drew@math.mit.edu  }
\thanks{$^\ast$Correspondence: University of Bristol, celine.maistret@bristol.ac.uk}
\date{May 2026}

\theoremstyle{plain}
\newtheorem{theorem}{Theorem}[section]

\newtheorem{proposition}[theorem]{Proposition}
\newtheorem{corollary}[theorem]{Corollary}

\theoremstyle{definition}
\newtheorem{definition}[theorem]{Definition}

\theoremstyle{remark}
\newtheorem{remark}[theorem]{Remark}



\newcommand{\mr}[1]{MathSciNet:\,\href{https://mathscinet.ams.org/mathscinet-getitem?mr=#1}{MR#1}}

\newcommand{\Q}{\mathbb{Q}}
\newcommand{\Qbar}{\overline{\mathbb{Q}}}
\newcommand{\Z}{\mathbb{Z}}
\newcommand{\F}{\mathbb{F}}

\renewcommand{\P}{\mathbb{P}}
\DeclareMathOperator{\Jac}{Jac}
\DeclareMathOperator{\End}{End}

\DeclareMathOperator{\rad}{rad}

\newcommand{\reductiontable}{
\begin{table}[ht!]
\begin{tabular}{c||>{\centering\arraybackslash}m{2.1cm}|c||>{\centering\arraybackslash}m{2.1cm}|>{\centering\arraybackslash}m{3.4cm}|c||c}
Prop. &$\mathcal{E}_p$    &$N_E$   &$\mathcal{C}_p$   &stable reduction of $C$  &$N_C$  &$N_A$  \\ \hline

\ref{prop:everything-good}

&\begin{tikzpicture}
\draw (0,0)--(2,0);
\node[circle,draw,fill=black,inner sep=0pt,minimum size=3pt] at (0.4,0) {};
\node[circle,draw,fill=black,inner sep=0pt,minimum size=3pt] at (0.8,0) {};
\node[circle,draw,fill=black,inner sep=0pt,minimum size=3pt] at (1.2,0) {};
\node[circle,draw,fill=black,inner sep=0pt,minimum size=3pt] at (1.6,0) {};
\node at (1.0,-0.25) {g1};
\node at (1.0,0.3) {};
\end{tikzpicture}

&0

&\begin{tikzpicture}
\draw (0,0)--(2,0);
\node at (1.0,-0.25) {g3};
\node at (1.0, 0.3) {};
\end{tikzpicture}

&good  &0  &0 \\ \hline

\ref{prop:pair}

&\begin{tikzpicture}
\draw (0,0)--(2,0);
\node[circle,draw,fill=black,inner sep=0pt,minimum size=3pt] at (0.4,0) {};
\node[circle,draw,fill=black,inner sep=0pt,minimum size=3pt] at (0.8,0) {};
\node[circle,draw,fill=black,inner sep=0pt,minimum size=3pt] at (0.4,-1) {};
\node[circle,draw,fill=black,inner sep=0pt,minimum size=3pt] at (0.8,-1) {};
\node at (1.0,-0.25) {g1};
\node at (1.0,0.15) {};
\draw (0,-1)--(2,-1);
\draw (1.8,0.10)--(1.4,-0.35);
\draw (1.4,-0.7)--(1.8,-1.1);
\draw[densely dotted] (1.5,-0.14)--(1.5,-.93);
\end{tikzpicture}

&0

&\begin{tikzpicture}
\draw (0,0)--(2,0);
\node at (1.0,-0.2) {g2};
\node at (1.0, 0.15) {};
\draw (1.8,0.10)--(1.4,-0.35);
\draw (0.2,0.10)--(0.6,-0.35);
\draw (0.6,-0.7)--(0.2,-1.1);
\draw (1.4,-0.7)--(1.8,-1.1);
\draw (0,-1)--(2,-1);
\draw[densely dotted] (1.5,-0.14)--(1.5,-.93);
\draw[densely dotted] (0.5,-0.14)--(0.5,-.93);
\end{tikzpicture}

&genus 2 with \quad\quad\quad\quad 1 self-intersection &1 &1 \\ \hline

\ref{prop:three-plus-one}

&\begin{tikzpicture}
\draw (0,0)--(2,0);
\node[circle,draw,fill=black,inner sep=0pt,minimum size=3pt] at (0.8,0) {};
\node[circle,draw,fill=black,inner sep=0pt,minimum size=3pt] at (1.2,-1) {};
\node[circle,draw,fill=black,inner sep=0pt,minimum size=3pt] at (0.4,-1) {};
\node[circle,draw,fill=black,inner sep=0pt,minimum size=3pt] at (0.8,-1) {};
\node at (1.1,-0.2) {g1};
\node at (1.0,0.15) {};
\draw (0,-1)--(2,-1);
\draw (1.9,0.10)--(1.5,-0.35);
\draw (1.5,-0.7)--(1.9,-1.1);
\draw[densely dotted] (1.6,-0.14)--(1.6,-.93);\begin{footnotesize}
\node at (0.73,-0.55) {even length};
\end{footnotesize}
\end{tikzpicture}

&0

&\begin{tikzpicture}
\draw (0,0)--(2,0);
\node at (0.8,-0.2) {g2};
\node at (1.0, 0.15) {};
\draw[line width=1.2pt] (1.5,0.1)--(1.7,-0.35);
\draw (1.7,-0.15)--(1.5,-0.55);
\draw[line width=1.2pt] (1.5,-0.7)--(1.7,-1.1);
\draw (0,-1)--(2,-1);
\draw[densely dotted] (1.55,-0.3)--(1.55,-0.95);
\node at (0.8,-0.8) {g1};
\end{tikzpicture}

&genus 2 and genus 1 intersecting in a node &0 &0 \\ \hline

\ref{prop:quadruple}(i)

&\begin{tikzpicture}
\draw (0,0)--(2,0);
\node[circle,draw,fill=black,inner sep=0pt,minimum size=3pt] at (0.3,-1) {};
\node[circle,draw,fill=black,inner sep=0pt,minimum size=3pt] at (0.6,-1) {};
\node[circle,draw,fill=black,inner sep=0pt,minimum size=3pt] at (0.9,-1) {};
\node[circle,draw,fill=black,inner sep=0pt,minimum size=3pt] at (1.2,-1) {};
\node at (0.8,-0.2) {g1};
\node at (1.0,0.15) {};
\draw (1.8,0.10)--(1.4,-0.35);
\draw (1.4,-0.7)--(1.8,-1.1);
\draw[densely dotted] (1.47,-0.14)--(1.47,-.93);
\draw (0,-1)--(2,-1);
\node[circle,draw,inner sep=0pt,fill=white,minimum size=3pt] at (1.6,-0.125) {};
\begin{footnotesize}
\node at (1.75,-0.25) {$Q$};
\end{footnotesize}
\end{tikzpicture}

&0

&\begin{tikzpicture}
\node at (0,0.12) {};
\draw (0,0)--(2,0);
\node at (0.3,-0.2) {g1};
\draw (0,-.75)--(2,-.75);
\node at (1.3,-0.55) {g1};
\draw (1.5,0.1)--(1.8,-0.2);
\draw (1.8,-.55)--(1.5,-.85);
\draw[densely dotted] (1.75,-0.06)--(1.75,-0.69 );
\draw (0.3,-0.65)--(0.6,-0.95);
\draw (0.6,-1.3)--(0.3,-1.6);
\draw[densely dotted] (0.55,-0.81)--(0.55,-1.44 );
\draw (0,-1.5)--(2,-1.5);
\node at (1.0,-1.3) {g1};
\end{tikzpicture}

&three genus 1s \quad linked in a chain &0 &0 \\ \hline

\ref{prop:quadruple}(ii)

&\begin{tikzpicture}
\draw (0,0)--(2,0);
\node[circle,draw,fill=black,inner sep=0pt,minimum size=3pt] at (0.3,-1) {};
\node[circle,draw,fill=black,inner sep=0pt,minimum size=3pt] at (0.6,-1) {};
\node[circle,draw,fill=black,inner sep=0pt,minimum size=3pt] at (0.9,-1) {};
\node[circle,draw,fill=black,inner sep=0pt,minimum size=3pt] at (1.2,-1) {};
\node[circle,draw,inner sep=0pt,fill=white,minimum size=3pt] at (0.6,0) {};
\node at (1.1,-0.2) {g1};
\node at (1.0,0.15) {};
\draw (1.8,0.10)--(1.4,-0.35);
\draw (1.4,-0.7)--(1.8,-1.1);
\draw[densely dotted] (1.5,-0.14)--(1.5,-.93);
\draw (0,-1)--(2,-1);\begin{footnotesize}
\node at (0.4,-0.2) {$Q$};
\end{footnotesize}
\end{tikzpicture}

&0

&\begin{tikzpicture}
\draw (0,0)--(2,0);
\node at (1.0,-0.2) {g1};
\node at (1.0, 0.15) {};
\draw (1.8,0.10)--(1.4,-0.35);
\draw (0.2,0.10)--(0.6,-0.35);
\draw (0.6,-0.7)--(0.2,-1.1);
\draw (1.4,-0.7)--(1.8,-1.1);
\draw (0,-1)--(2,-1);
\draw[densely dotted] (1.5,-0.14)--(1.5,-.93);
\draw[densely dotted] (0.5,-0.14)--(0.5,-.93);
\node at (1.0,-0.8) {g1};
\end{tikzpicture}

&two genus 1s intersecting twice &1 &1 \\ \hline

\ref{prop:multiplicative}(i)

&\begin{tikzpicture}
\draw (0,0)--(2,0);
\node[circle,draw,fill=black,inner sep=0pt,minimum size=3pt] at (0.2,0) {};
\node[circle,draw,fill=black,inner sep=0pt,minimum size=3pt] at (0.6,0) {};
\node[circle,draw,fill=white,inner sep=0pt,minimum size=3pt] at (1,0) {};
\node[circle,draw,fill=black,inner sep=0pt,minimum size=3pt] at (1.4,0) {};
\node[circle,draw,fill=black,inner sep=0pt,minimum size=3pt] at (1.8,0) {};
\node at (1.0, 0.15) {};
\draw (1.6,0.1)--(1.8,-0.7);
\draw (1.8,-0.3)--(1.6,-1.1);
\draw (0.4,0.1)--(0.2,-0.7);
\draw (0.2,-0.3)--(0.4,-1.1);
\draw (0.2,-0.9)--(0.7,-1.2);
\draw (1.8,-0.9)--(1.3,-1.2);
\draw[densely dotted] (0.5,-1.15)--(1.5,-1.15);
\begin{footnotesize}
\node at (1.0,-0.25) {$Q$};
\end{footnotesize}
\end{tikzpicture}

&1

&\begin{tikzpicture}
\draw (0,1)--(2,1);
\node at (1.0,1.2) {g1};
\draw (0.2,0.9)--(0.3,1.5);
\draw (0.2,1.3)--(0.6,1.7);
\draw (1.6,0.9)--(1.5,1.5);
\draw (1.6,1.3)--(1.2,1.7);
\draw (1.8,1.1)--(1.7,0.5);
\draw (1.8,0.7)--(1.4,0.3);
\draw (0.4,1.1)--(0.5,0.5);
\draw (0.4,0.7)--(0.8,0.3);
\draw[densely dotted] (0.4,1.6)--(1.4,1.6);
\draw[densely dotted] (0.6,0.4)--(1.6,0.4);
\node at (0,1.72) {};
\end{tikzpicture}

&genus 1 with \quad\quad\qquad 2 self-intersections  &2  &1 \\ \hline

\ref{prop:multiplicative}(ii)

&\begin{tikzpicture}
\draw (0,0)--(2,0);
\node at (1.0, 0.05) {};
\draw (1.8,0.10)--(1.4,-0.35);
\draw (0.2,0.10)--(0.6,-0.35);
\draw (0.6,-0.7)--(0.2,-1.1);
\draw (1.4,-0.7)--(1.8,-1.1);
\draw (0,-1)--(2,-1);
\draw[densely dotted] (1.5,-0.14)--(1.5,-.93);
\draw[densely dotted] (0.5,-0.14)--(0.5,-.93);
\node[circle,draw,fill=black,inner sep=0pt,minimum size=3pt] at (0.5,0) {};
\node[circle,draw,fill=black,inner sep=0pt,minimum size=3pt] at (0.825,0) {};
\node[circle,draw,fill=white,inner sep=0pt,minimum size=3pt] at (1,-1) {};
\node[circle,draw,fill=black,inner sep=0pt,minimum size=3pt] at (1.175,0) {};
\node[circle,draw,fill=black,inner sep=0pt,minimum size=3pt] at (1.5,0) {};
\begin{footnotesize}
\node at (1.0,-0.75) {$Q$};
\end{footnotesize}
\end{tikzpicture}

&1

&\begin{tikzpicture}
\draw (0,0)--(2,0);
\node at (1.0, 0.15) {};
\draw[line width=1.2pt] (1.5,0.1)--(1.8,-0.5);
\draw (1.8,-0.15)--(1.6,-0.9);
\draw[line width=1.2pt] (0.5,0.1)--(0.2,-0.5);
\draw (0.2,-0.15)--(0.4,-0.9);
\draw[line width=1.2pt] (0.2,-0.7)--(0.7,-1.0);
\draw[line width=1.2pt] (1.8,-0.7)--(1.3,-1.0);
\draw[densely dotted] (0.5,-0.95)--(1.5,-0.95);
\node at (1.0,-0.2) {g2};
\end{tikzpicture}

&genus 2 with \qquad\qquad 1 self-intersection  &1 &0
\end{tabular}\\[0.12cm]
\caption{Results from \Cref{sect:construction} for odd primes.}\label{table:reductions}
\end{table}}

\begin{document}

\begin{abstract}
We describe the construction of a database of roughly half a million abelian surfaces over $\Q$ of small conductor arising as Pryms associated to a genus 3 double cover of a genus 1 curve. Our construction uses a method that provides a degree of control over the primes of bad reduction.
\end{abstract}


\maketitle 

\tableofcontents 

\section{Introduction}

Tables of elliptic curves of small conductor compiled by Cremona \cite{cremona} and others are now a key component of the $L$-functions and modular forms database (\href{https://www.lmfdb.org}{LMFDB}) \cite{LMFDB}.
These tables have a long history of service to the number theory community, and they have proven to be a rich source of problems and applications for computational number theorists.
The fact that the conductors are small is crucial to the utility of this dataset in the LMFDB, as it makes it feasible to explicitly match them to modular forms (via the modularity theorem \cite{BCDT}) and compute their $L$-functions.
In 2015 a database of 66\,158 genus~2 curves $C/\Q$ with minimal discriminant $|\Delta_{\min}|\le 10^6$ was added to the LMFDB \cite{BSSVY}.
These small discriminant curves necessarily have (Jacobians with) small conductors, but most small conductor curves do not have small discriminants.
This is due in large part to the presence of primes of \emph{almost good reduction} (see Definition~\ref{def:almost-good}), which divide the discriminant but not the conductor; these primes can be much larger than the conductor itself, as can be seen in the examples of \cite{MS25}.

More recently a larger database of genus~2 curves $C/\Q$ of small conductor was compiled by Booker and the fourth author \cite{BS26,SutRatpointsTalk}, and is now available in the \href{https://alpha.lmfdb.org/Genus2Curve/Q/}{alpha release of the LMFDB} \cite{alpha}.
This includes more than 6 million genus~2 curves of conductor $N\le 10^6$, more than 90 times as many as are known with $|\Delta_{\min}|\le 10^6$.
But this new database of genus~2 curves necessarily omits many abelian surfaces of small conductor, because most abelian surfaces over $\Q$ do not arise as a Jacobian.
Products of elliptic curves already account for more than 27 million distinct isogeny classes of abelian surfaces over $\Q$ of conductor less than $10^6$, and Weil restrictions of elliptic curves defined over quadratic fields account for many more.

Extensive tables of elliptic curves over $\Q$ and quadratic fields can be found in the LMFDB. Provided one is willing to assume modularity (now known for $\Q$ and all real quadratic fields, as well as many imaginary quadratic fields), there are rigorous methods to enumerate all such elliptic curves up to a given conductor bound.
But this is not the case for geometrically simple abelian surfaces, and after the recent heuristic compilation of genus-2 curves of small conductor, the next most interesting case to consider is geometrically simple abelian surfaces over $\Q$ that are not principally polarised (hence not Jacobians of curves), such as those that arise as Prym varieties associated to a genus 3 cover of an elliptic curve.
The necessity of considering such examples was demonstrated by the abelian surface of conductor $550$ discovered in the process of constructing a database of genus 3 nonhyperelliptic curves \cite{Sut19}, which arises as the Prym variety associated to a smooth plane quartic curve that admits a degree-2 map to an elliptic curve.
It corresponds to one of nine weight-2 paramodular forms of level $N\le 1000$ in the heuristic tables of Poor and Yuen \cite{PSY20,PY15,PY20,PY} for which no corresponding Jacobian is known.
Its existence was predicted by the paramodular conjecture of Brumer and Kramer \cite{BK14,BK19}.

This motivates our current work: a heuristic algorithm to construct abelian surfaces of small conductor that arise as Prym varieties associated to double covers of an elliptic curve.
A key feature of our approach, unlike the methods used to construct databases of genus 2 curves, is our ability to control, to a certain extent, the primes of bad reduction.
This not only allows us to construct many new geometrically simple abelian surfaces of small conductor, it also allows us to construct many Jacobians of genus 2 curves of small conductor that were not previously known.
This notably includes ten genus 2 curves whose Jacobians have good reduction away from 2, extending the list of 512 curves found by Robin Visser \cite{Visser}; see Section~\ref{sect:postscript} for a list of these curves. 

We should emphasise that our algorithm is necessarily heuristic, and we can make no claims of completeness.
For any fixed dimension $g$ and set of primes $S$, the set of abelian varieties of dimension $g$ over $\Q$ with good reduction away from $S$ is known to be finite (this is Shafarevich's conjecture, proved by Faltings \cite{Faltings}), but for $g>1$ no algorithm is known to enumerate this set.
Indeed, this is a major open problem whose solution would yield an algorithm for determining the set of rational points on a curve of genus $g\ge 2$ over $\Q$, as shown in \cite{AL}.   We have the more modest goal of constructing  many (but not necessarily all) abelian surfaces over $\Q$ with good reduction away from a specified set of primes. 

To construct such abelian surfaces as Prym varieties, we study degree-2 covers of genus 1 curves.
Our approach is to start with an elliptic curve~$E$ over $\Q$ and prescribe the branch points of the putative cover.

To do so, we pick a rational effective divisor $D$ on $E$ of degree 4, such that all prime divisors in $D$ occur with multiplicity 1.
Assuming there exists a rational divisor $D'$ such that $D \sim 2D'$ in the divisor class group of~$E$, there exists a rational function $f$ with $\mathrm{div}(f) = D - 2D'$. The cover $\phi \colon C \to E$ is obtained by adjoining the square root of $f$ to the function field of~$E$, and by Riemann--Hurwitz, $C$ is a curve of genus 3.
We denote the Prym variety associated to this cover $\phi$ by~$A$.

\begin{remark}\label{rmk:non-unique}
Note that the divisor $D$ does not determine the cover $\phi$.
There are two degrees of freedom.
The divisor $D'$ can be replaced by $D' + T$ for a 2-torsion element $T \in \mathrm{Pic}^0(E)$.
If~$T \neq 0$, this gives rise to a different curve $C'$ that generally is not isomorphic to~$C$, even over $\Qbar$; see also Remark \ref{rmk:2-torsion}.
It is also possible to replace the function $f$ by~$d \cdot f$ for some~$d \in \Q^*$.
This gives a quadratic twist of the original cover $\phi$.  We can check for quadratic twists of small conductor in a post-processing step.
\end{remark}

Similar to the case of hyperelliptic curves, where the stable reduction (i.e., the special fibre of a stable model) is understood through the degeneration of the branch points of the map to $\P^1$ (see \cite{Kausz}), we study the stable reduction of the genus 3 cover through the degeneration of the branch points on the genus~1 curve.
The method is more subtle in the case of the genus 3 cover, as the stable reduction of the cover is not entirely determined by the degeneration of the branch locus alone.
For example, in the case in which all the branch points coincide (see \Cref{prop:quadruple} or the corresponding rows in \Cref{table:reductions}), it cannot be determined from the branch locus alone whether the stable reduction has one or two connected components above the genus 1 curve.
Indeed, it does not suffice to determine the degeneration of the branch locus; the degeneration of the whole divisor $D - 2D'$ must be considered.

\reductiontable

This is detailed in \Cref{sect:construction}, where we carefully curated some basic cases for which the Prym variety has conductor exponent 0 or 1. The cases chosen have relatively simple degenerations of the branch points, meaning they are more likely to be found in a search while also giving rise to a small conductor exponent in the Prym.
The results for odd primes $p$ are summarised in \Cref{table:reductions}.
In this table, $N_E$, $N_C$, and $N_A$ are the conductor exponents of $E$, $\mathrm{Jac}(C)$, and $A$, respectively, $\mathcal{C}_p$ is the special fibre of a regular model of~$C$, $\mathcal{E}_p$ is the special fibre of a regular model of~$E$ that separates the branch points of the cover, black points indicate branch points, thick lines indicate components of multiplicity 2, dotted lines indicate chains of~$\P^1$s, and all components have geometric genus 0 and multiplicity 1 unless otherwise indicated.
For the prime~$p = 2$, we give a sufficient criterion for $C$ and the Prym to have good reduction at 2, see \Cref{prop:evenprime}.
We note that these results are far from exhaustive.
However, the same method can be used to study many other possible degenerations in the case that $p$ is odd. As with hyperelliptic curves, the criterion for $p = 2$ is rather ad hoc and we do not expect it to generalise easily.

The rest of the paper is organised as follows. In \Cref{sect:construction} we cover the mathematical background and technical details of the double covers we construct.  \Cref{sect:conductor} addresses the problem of computing the conductor of the Pryms associated to these double covers, and describes an efficient heuristic algorithm for doing so.  \Cref{sect:database} discusses the computations we ran to construct many Pryms of small conductor, and \Cref{sect:results} summarises our computational results, which include the enumeration of more than 100 million genus 3 double covers, among which we found nearly half a million distinct isogeny classes of Prym varieties with conductors up to $2^{20}$, of which more than 80\,000 were not previously known.  This includes more than 10\,000 that are isogenous to a genus 2 Jacobian that was not previously known.

\subsection{Acknowledgements}

We would like to warmly thank the referees for their careful reading of this manuscript and their helpful suggestions and comments. The first and second authors were supported by the second author's Royal Society Dorothy Hodgkin Fellowship. The third and fourth authors were supported by Simons Foundation award 550033. The third author was also supported by the Natural Sciences and Engineering
Research Council of Canada (NSERC). We would like to thank Tim Dokchitser and Gergely Jakovac for their help computing models of curves in the experiments that led to the results in this paper.  We also thank Andrew Booker for providing the code we used to test the functional equation, and Nils Bruin for helpful comments.

\newcommand{\Zpunr}{\Z_p^{\mathrm{unr}}}

\section{Pryms with prescribed local data}
\label{sect:construction}

\subsection{Generalities}

Recall from the introduction that the genus 3 cover $\phi \colon C \to E$ was constructed by adjoining the square root of a function $f$ whose divisor is~$D - 2D'$.
First we show that $C$ has tame reduction at all primes $p> 3$.

\begin{proposition}\label{prop:tame-primes}
Let $\phi \colon C \to E$ be a degree $2$ map from a genus $3$ curve $C$ to a genus $1$ curve $E$, all defined over $\Q$.
Let $p > 3$ be a prime number.
Then $C$ has tame reduction at $p$.
\end{proposition}

\begin{proof}
Because $E$ is of genus 1, $E$ attains stable reduction over a field extension~$K / \Q$ that is at most tamely ramified at $p$ since $p > 2\cdot \mathrm{genus}(E) + 2 = 4$.
Let $B \subset E$ be the branch locus of the map $\phi$, and let $K(B)$ be the Galois extension of $K$ generated by the coordinates of the points in $B$.
The Galois group $\mathrm{Gal}(K(B)/K)$ acts as a subgroup of $S_{|B|} = S_4$, and therefore $K(B) / K$ is at most tamely ramified at all primes above $p$. 
By \cite[Thm.\ 2.3]{LiuLorenzini}, the curve $C$ has stable reduction over an extension field $L / K(B)$ of degree at most 2.
This extension is also at most tamely ramified at all primes above $p$, and hence $p$ is a prime of tame reduction for~$C$.
\end{proof}

\begin{remark}\label{remark:conductor-Prym}
The conductor of the Prym variety is the quotient of the conductor of the Jacobian of the genus 3 curve by the conductor of $E$.
Therefore, by \Cref{prop:tame-primes}, to understand the conductor of the Prym variety at primes~$p > 3$, it suffices to understand the tame conductor exponent of $C$ and the conductor exponent of $E$ at these primes.
\end{remark}

\subsection{Local data at odd primes}

Let $C, E, \phi, f, D$, and $D'$ be as before.
In order to understand the geometric properties of $C$ at a prime $p > 2$, we look at how the divisor $D - 2D'$ degenerates modulo powers of~$p$.
For the purpose of this work, we will assume that $E$ has stable reduction at $p$, even though it would be possible to also study the case in which $E$ does not have stable reduction using this method.
As discussed in the introduction, the approach is very similar to that of hyperelliptic curves.
In the hyperelliptic case, $D$ consists of the images of the Weierstra\ss\ points on $\P^1$, and $D'$ is $(g+1) \cdot \infty$, and it only matters how~$D$ degenerates. In our case it also matters how~$D'$ degenerates.
The following is an adaptation of a method for hyperelliptic curves that is explicitly described in \cite[Sect.\ 4]{Kausz}, and can be found in more generality in \cite[Lemma~1.9]{LiuLorenzini}.

We construct a tree-like regular model $\mathcal{E} / \Zpunr$ that separates all the points in $D$ and $D'$ by repeatedly blowing up, starting from a semi-stable regular model of $E$ over $\Zpunr$.
The function $f$ gives rise to a rational function~$\tilde{f} \in K(\mathcal{E})$, because $E$ and $\mathcal{E}$ have the same function field.

Similar to \cite[Sect.\ 4]{Kausz}, we define a divisor $D_{\mathrm{odd}}$ (called $C$ in loc.\ cit.) which captures the odd part of $\mathrm{div}(\tilde{f})$, and a divisor $B$ such that $2B = D_\mathrm{odd} - \mathrm{div}(\tilde{f})$.
Then we construct a cover $\mathcal{C}$ of $\mathcal{E}$ by taking $\mathrm{Spec}(\mathcal{O}_\mathcal{E} \oplus \mathcal{L}^{-1})$, where the bundle $\mathcal{L}$ is defined as $\mathcal{O}_{\mathcal{E}}(B)$.
The model $\mathcal{C} / \Zpunr$ of $C$ is then a regular model, see \cite[Prop.\ 4.3]{Kausz}.

Note that $\mathrm{div}(\tilde{f})$ can also contain vertical divisors of $\mathcal{E}$.
The divisor $\mathrm{div}(\tilde{f})$ must have the property that the intersection product satisfies $\mathrm{div}(\tilde{f}) \cdot V = 0$ for any vertical divisor $V$ of $\mathcal{E}$.
This uniquely determines the vertical part of $\mathrm{div}(\tilde{f})$ up to multiples of the whole special fibre $\mathcal{E}_p$ of $\mathcal{E}$, see \cite[Thm.\ III.3.6]{Lang88}.
In \cite[Sect.\ 4]{Kausz}, this is made explicit in terms of the number of points in $D$ and $D'$ whose reduction lies in a certain vertical component or in a lower component in the tree.
If $\tilde{f}$ has odd order at a vertical component of $\mathcal{E}$, then the unique component above it will have multiplicity~2, and if $\tilde{f}$ has even order then the components above it will have multiplicity~1.
Therefore, if there are no vertical components in $D_\mathrm{odd}$, then the model $\mathcal{C}$ is semi-stable.
If not, then $\mathcal{C}$ will become semi-stable after a ramified extension of degree~2, see for example \cite[Thm.\ 2.3]{LiuLorenzini}.

The following propositions list some configurations for which the conductor exponent of $C$ (and hence $A$) is small.
A pictorial summary of these can be found in \Cref{table:reductions}.

\begin{proposition}\label{prop:everything-good}
Suppose that $E$ has good reduction at $p$.
If the four points in $D$ reduce to distinct points in $E \bmod p$, then, after a quadratic twist by $p$ if necessary, $C$ has good reduction at $p$.
\end{proposition}

\begin{proof}
In this case, no blow-ups are necessary and $\mathcal{E}$ can be taken to be a smooth model of $E$.
If $\widetilde{f}$ happens to have odd order at the vertical divisor, then we do a quadratic twist.
That is, we replace $f$ by $p \cdot f$, and the resulting function has even order at the vertical divisor.
Now the constructed cover $\mathcal{C}$ is regular and semi-stable.
The special fibre $\mathcal{C}_p$ is a degree-2 cover, branched at a degree 4 divisor on~$\mathcal{E}_p$, i.e., $\mathcal{C}_p$ is a smooth curve of genus 3.
\end{proof}

\begin{proposition}\label{prop:pair}
Suppose that $E$ has good reduction at $p$.
If exactly two points in $D$ reduce to the same point in $E \bmod p$, then, after a quadratic twist by $p$ if necessary, $C$ has stable reduction at $p$, with special fibre being a genus 2 curve with a self-intersection.
\end{proposition}

\begin{proof}
Let us first assume the points occurring in $D$ are defined over $\Q_p^{\mathrm{unr}}$.
Then $\mathcal{E}$ is obtained by blowing up a smooth model several times, and its special fibre is a genus 1 curve with two branch points, and attached to that genus 1 curve, there is a chain of $\P^1$s with the two other branch points at the end (as in \Cref{table:reductions}).
Since there is an even number of branch points on each vertical component, $\tilde{f}$ either has even order at all vertical components or odd order at all vertical components.
We multiply $f$ again by the appropriate power of $p$ to make sure it has even order at all vertical components.
Using Riemann--Hurwitz, we see that the special fibre $\mathcal{C}_p$ consists of a genus 2 component lying above the genus 1 curve in $\mathcal{E}_p$, and a loop of $\P^1$s attached to the genus 2 component.

Suppose the points in $D$ reducing to the same point in $E \bmod p$ are defined over a ramified degree~2 extension of $\Q_p^{\mathrm{unr}}$.
Then, after extending the base field to this degree-2 extension, $\tilde{f}$ has even order at all vertical components, and we get a stable model as in the previous case over this degree-2 extension.
If inertia mirrors the chain of $\P^1$s and hence acts non-trivially on the genus 2 component, we can multiply $\tilde{f}$ by $p$ to make inertia act trivially on the genus 2 component. Therefore, after twisting if necessary, $C$ has the stable reduction at $p$ as claimed.
\end{proof}

The following proposition is the first case in which we get a prime of almost good reduction. Let us first recall the definition of almost good reduction.

\begin{definition}\label{def:almost-good}
The curve $C$ is said to have {\em almost good reduction} at $p$ if $C$ has bad reduction at $p$ and the Jacobian of $C$ has good reduction at $p$. 
\end{definition}

This happens when the curve has stable reduction and the special fibre of the stable model has no loops. In this case, the reduction of the Jacobian is isogenous to the product of the Jacobians of the irreducible components of this special fibre. As a consequence, $C$ has conductor exponent 0 at $p$, even though $C$ has bad reduction at $p$. 

\begin{proposition}\label{prop:three-plus-one}
Suppose that $E$ has good reduction at $p$ and that all points in $D$ are defined over $\Q_p^\mathrm{unr}$.
Let $n$ be a positive integer and assume that three of the points in $D$ reduce to the same point in $E \bmod p^{2n}$, but no two of them reduce to the same point in $E \bmod p^{2n+1}$. Suppose furthermore that the fourth point reduces to a different point in $E \bmod p$. Then, after a quadratic twist by $p$ if necessary, $C$ has almost good reduction at $p$, with special fibre consisting of a genus $2$ curve intersecting a genus 1 curve.
\end{proposition}

\begin{proof}
After twisting if necessary, the special fibre of $\mathcal{C}$ becomes a genus 2 curve connected with a genus 1 curve through a chain of $\P^1$s which alternately have multiplicity 2 and 1.
The $\P^1$s with multiplicity 2 in the special fibre can be contracted by Castelnuovo's criterion, obtaining a regular semi-stable model, which gives the desired reduction.

Note that in the twisting that we just did, we used that $2n$ is even.
Indeed, if the three points reduce to the same point modulo an odd power of $p$, then the chain of $\P^1$s has even length.
This means that the cover is ramified at either the vertical genus 1 component or the final vertical genus 0 component with the three branch points.
In this case, twisting will not make the cover unramified, as it will only swap the components at which the cover is ramified with the ones at which it is unramified.
\end{proof}

\begin{proposition}\label{prop:quadruple}
Suppose that $E$ has good reduction at $p$, $D = P_1 + P_2 + P_3 + P_4$ and $D' = P_4 + Q$, where $P_1, \ldots, P_4, Q$ are defined over $\Q_p^\mathrm{unr}$.
Moreover, assume that $P_1, P_2, P_3, P_4$ reduce to the same point on $E$ modulo $p^n$, for some positive integer $n$, such that no two reduce to the same point modulo $p^{n+1}$.
\begin{itemize}
\item[(i)]
If $Q$ also reduces to that point modulo $p$, then, after a quadratic twist if necessary, $C$ has almost good reduction, with special fibre consisting of three genus 1 curves in a chain.
\item[(ii)]
If $Q$ reduces to a different point modulo $p$, then, after a quadratic twist if necessary, $C$ has stable reduction, with special fibre consisting of two genus 1 curves intersecting twice.
\end{itemize}
\end{proposition}

\begin{proof}
In both cases, after a quadratic twist if necessary, the cover $\mathcal{C} \to \mathcal{E}$ does not have any vertical components in its branch divisor.
It remains to understand the cover above the vertical genus 1 component.
Since the cover has no ramification points on that component, there are two possibilities: a genus 1 curve mapping to the genus 1 curve through a 2-isogeny, or two disjoint copies of the genus 1 curve.
To distinguish these two cases we consider the restriction of the divisor $\mathrm{div}(\tilde{f})$ to the vertical genus 1 component.

If $Q$ reduces to the same point as $P_1,P_2,P_3,P_4$, then the restriction of $\mathrm{div}(\tilde{f})$ to the vertical genus 1 component is 0, and the cover is locally given by adjoining the square root of a constant function, i.e., the cover consists of two disjoint genus 1 curves.

If $Q$ reduces to a different point, then the restriction of $\mathrm{div}(\tilde{f})$ to the vertical genus 1 component is $2\overline{P_1} - 2\overline{Q}$, where $\overline{P_1}$ is the intersection point with the chain of $\P^1$s leading to the branch points.
In this case $\tilde{f}$ does not restrict to a constant function on the vertical genus 1 component, and the cover is a 2-isogeny.

In both cases, the chain of $\P^1$s can be contracted to obtain the desired stable reduction.
\end{proof}

\begin{proposition}\label{prop:multiplicative}
Suppose that $E$ has multiplicative reduction at $p$.
Let $\mathcal{E}$ be its minimal regular (semi-stable) model.
Suppose that $D = P_1 + P_2 + P_3 + P_4$ and $D' = P_4 + Q$, with $P_1, \ldots, P_4, Q$ defined over $\Q_p^{\mathrm{unr}}$.
Moreover, assume that $P_1, P_2, P_3, P_4$ reduce to distinct smooth points on the same irreducible component of $\mathcal{E}_p$.
\begin{itemize}
\item[(i)] If $Q$ reduces to a smooth point on the same irreducible component of $\mathcal{E}_p$, then, after a quadratic twist if necessary, $C$ has stable reduction, with special fibre consisting of a genus 1 curve with two self-intersections.
\item[(ii)] If not, then, after a quadratic twist if necessary, $C$ has stable reduction, with special fibre consisting of a genus 2 curve with a self-intersection.
\end{itemize}
\end{proposition}

\begin{proof}
The special fibre of $\mathcal{E}$ is an $n$-gon for some positive integer $n$, with components $C_1, \ldots, C_n$ in that order.
Without loss of generality, we assume that the reductions of $P_1, \ldots, P_4$ lie on $C_n$.
Note that the reduction of $Q$ is a smooth point, because $Q$ is defined over $\Q_p^{\mathrm{unr}}$.
Because $D - 2D' \sim 0$ on $E$, this implies that $Q$ can only be on the same component as $P_1, \ldots, P_4$, or on the exact opposite component $C_{n/2}$ in case $n$ is even.
In particular, if $n$ is odd, we are always in case (i).

In case (i), we have that the vertical part of $\mathrm{div}(\tilde{f})$ is $0 + m \cdot \sum_{i=1}^n C_i$ for some~$m \in \Z$.
Multiplying $f$ by $p^{-m}$, twisting the cover if necessary, we may assume that $\mathrm{div}(\tilde{f})$ has no vertical part.
Then the cover is ramified above $C_n$ with four ramification points, and unramified above all other components $C_i$.
Hence, $\mathcal{C}_p$ consists of a genus 1 curve with 2 loops of $\P^1$s, and there is stable reduction as claimed.

In case (ii), we have that the vertical part of $\mathrm{div}(\tilde{f})$ is $$-(C_1 + 2C_2 + 3C_3 + \ldots + \tfrac{n}{2}C_{\frac{n}{2}} + (\tfrac{n}{2} - 1) C_{\frac{n}{2}+1} + \ldots + 2C_{n-2} + C_{n-1}) + m \cdot \sum_{i=1}^n C_i,$$
for some $m \in \Z$, which we again can assume to be 0, twisting the cover if necessary.
Now the cover above the components is:
\begin{itemize}
\item ramified at six points on $C_n$ (the points $P_1, P_2, P_3, P_4$ and the intersection points with the adjacent components), giving rise to a genus 2 curve in $\mathcal{C}_p$;
\item fully ramified above the odd-numbered components $C_{2k+1}$, giving rise to a genus 0 component with multiplicity 2 on $\mathcal{C}_p$ above each such component;
\item ramified at two points on the other even-numbered components $C_{2k}$, giving rise to a genus 0 component with multiplicity 1 on $\mathcal{C}_p$  above each such component.
\end{itemize}

Just as in the proof of \Cref{prop:three-plus-one}, the components of multiplicity 2 can be contracted due to Castelnuovo's criterion.
Hence, we get a genus 2 curve with a loop of $\P^1$s, and there is stable reduction as claimed.
\end{proof}

\subsection{Local data at the even prime}

Since the cover is of degree 2, the prime $p = 2$ behaves differently.
Recall that an elliptic curve given by $y^2 = x^3 - n$ has good reduction at 2 if and only if $v_2(n) \equiv 4 \bmod 6$ and $\frac{n}{2^{v_2(n)}} \equiv 3 \bmod 4$, see for example \cite[Table~1]{BGR}.
Indeed, in this case, changing coordinates if necessary, we may assume that $v_2(n) = -2$ and $n \equiv \tfrac34 \bmod 1$.
Replacing $y$ by $y + \tfrac12$ then yields an integral Weierstra\ss\ model with good reduction at $2$.
The next proposition is inspired by the fact that the four ramification points in this example have $2$-adic distance $2^{-\frac23}$.

\begin{proposition}\label{prop:evenprime}
Suppose that $E$ has good reduction at $2$.
Let $\mathcal{E}$ be a good Weierstra\ss\ model of $E$ over $\Z_2$ with projective coordinates $x,y,z$.
Suppose $D = P_1 + P_2 + P_3 + P_4$ with $P_1, P_2, P_3$ defined over $\overline{\Q_2}$, and fix $P_4 = 0_E$ and $D' = 2 \cdot 0_E$.
Let $v_2$ be the valuation on $\overline{\Q_2}$ normalised so that $v_2(2) = 1$.
Assume that $P_i = (x_i : 1 : z_i)$ with $v_2(x_i), v_2(z_i) > 0$, for all $i = 1,2,3$, and that the $2$-adic distances between the points are $2^{-\frac23}$, i.e.\ $\min(v_2(x_i - x_j), v_2(z_i - z_j)) = \tfrac23$ for all~$i \neq j$ in $\{1,2,3,4\}$.
Then, after twisting if necessary, the curve $C$ has good reduction at $2$.
\end{proposition}

\begin{proof}
The 2-adic distance between $P_i$ and $P_4$ is $2^{-\frac23}$. By definition, this means that $\min(v_2(x_i), v_2(z_i)) = \frac23$.
Let 
$$y^2z + a_1 xyz + a_3yz^2 = x^3 + a_2x^2z + a_4xz^2 + a_6z^3$$
be the integral Weierstra\ss\ model $\mathcal{E}$.
We consider the 2-adic valuations of the separate terms for the points $P_1, P_2, P_3$.
By the non-archimedean property, we find that $v_2(z_i) = v_2(x_i^3)$.
In particular, $v_2(z_i) = 2$ and $v_2(x_i) = \tfrac23$.

The function $f$ whose square root we adjoin to get the genus 3 curve lies in the Riemann--Roch space $\mathcal{L}(3 \cdot 0_E) = \langle 1, \frac{x}{z}, \frac{y}{z} \rangle$.
We normalise $f$ so that the coefficient of $\frac{y}{z}$ is 1, i.e.\ $f = \frac{y}{z} + c_1 \frac{x}{z} + c_2$.
Note that $f$ has zeros at $P_1, P_2, P_3$.
As $v_2(1) = 0$, $v_2(\frac{x_i}{z_i}) = -\tfrac43$, and $v_2(\frac{y_i}{z_i}) = - 2$, the term $c_1 \frac{x_i}{z_i}$ is the only term whose valuation is not an integer. We therefore must have~$-2=v_2(\frac{y_i}{z_i}) = v_2(c_2)$, and also $v_2(c_1 \frac{x_i}{z_i}) > -2$, that is, $v_2(c_1) \geq 0$.
In particular,~$c_2 \in \tfrac14 + \Z_2$ or $c_2 \in \tfrac34 + \Z_2$.
In the latter case, we multiply $f$ by $-1$ to obtain $c_2 \in \tfrac14 + \Z_2$.

Now the genus 3 curve $C$ obtained by adjoining a square root of $f$, after twisting if necessary, can be put into a long form $w^2 + w + \tfrac14 = f$, and by subtracting $\tfrac14$ on both sides we obtain a model with good reduction at 2.
\end{proof}

\subsection{Conductor exponent of the Prym}

By \Cref{remark:conductor-Prym}, the different situations that we studied above have immediate consequences for the conductor of the Prym variety associated to the cover $C \to E$.

\begin{corollary}
The conductor exponents presented in \Cref{table:reductions} are correct.
\end{corollary}

\begin{proof}
The conductor exponent $N_C$ of $C$ at $p$ is defined by the action of inertia on the Tate module, see for example \cite{BK94}.
In all cases considered in \Cref{table:reductions}, the stable reduction was obtained after at most a ramified extension of degree~2.
Therefore,~$N_C$ has no wild part, and we only need to determine its tame part.
Since $C$ has stable reduction in all cases, this translates into computing the dimension of the homology of the dual graph of~$\mathcal{C}_p$, i.e.\ the number of loops in~$\mathcal{C}_p$.
\end{proof}

In the case of \Cref{prop:evenprime}, where $p = 2$, the curve $C$ has good reduction at $p$ and hence the conductor exponents of $C$ and $A$ are 0.

\section{Computing the conductor of the Prym}
\label{sect:conductor}

The methods of the previous section allow us to control the primes of bad reduction for the Pryms we construct and give us some information about the conductor exponents at these primes, but not enough to determine the conductor of the Prym variety $A$ associated to a genus~3 double cover $\phi\colon C\to E$ of an elliptic curve $E$.
While in principle one can compute the conductor of $A$ as the quotient of the conductors of $\Jac(C)$ and $E$, this approach is impractical, as there are no efficient methods available for computing the conductor of the Jacobian of a general genus~3 curve over $\Q$.  Partial results are known in some special cases, such as when $C$ is a Picard curve \cite{BBW,BKSW}, a Ciani quartic \cite{BCLS}, a hyperelliptic curve \cite{Spencer}, or a plane curve \cite{LR25}, but even in these special cases the available methods are not fast enough for our purposes.

We instead adopt the heuristic approach used for existing databases of genus~2 and genus~3 curves over $\Q$ of small conductor \cite{BSSVY,Sut19}, including the database of genus 2 curves found in the LMFDB. We now assume that the $L$-function $L(A,s)=\sum_{n\ge 1} a_n n^{-s}$ of our abelian surface $A/\Q$ satisfies the expected functional equation
\begin{equation}\label{eq:func}
\Lambda(A,s) = \varepsilon N^{1-s}\Lambda(A,2-s),
\end{equation}
where $\varepsilon=\pm 1$ is the root number and $\Lambda(A,s)=\Gamma_\mathbb{C}(s)^2L(A,s)$ is the completed $L$-function of $A$ (here $\Gamma_\mathbb{C}(s)\coloneqq 2(2\pi)^{-s}\Gamma(s)$ with $\Gamma(s)\coloneqq \int_0^\infty t^{s-1}e^{-t}dt$).
As explained in \cite[\S 5.2]{BSSVY}, given a purported conductor $N$, root number $\varepsilon$, and Dirichlet coefficients~$a_n$ for $n\le B= b\sqrt{N}$, where $b\ge 5$ is a fixed constant, there is an explicit numerical test that must be satisfied if \eqref{eq:func} holds and these purported values are all correct, and is unlikely to be satisfied whenever this is not the case.

Increasing $b$ strengthens this test, but also requires us to compute more Dirichlet coefficients; we used $b=12$ for our computations.
For the sake of efficiency we first perform computations at standard double precision (53 bits) and accept the possibility of rounding errors that may arise from the use of floating-point arithmetic. After the fact, we can use interval arithmetic as implemented in \cite{Johansson} to rigorously verify our computations, subject to our assumption that \eqref{eq:func} holds.  This assumption is implied by the conjectured modularity of~$A$, but this conjecture remains open, so even with interval arithmetic, all our conductor computations are conditional.

For any fixed $N$ there are only finitely many possibilities for $\varepsilon$ and the $a_n$ for $n\le B$.
If our numerical test of \eqref{eq:func} fails for all possibilities, then we know that $N$ is not the conductor of $A$ (under our assumption).  For any particular genus 3 double cover $C\to E$, the prime divisors of $N$ must divide the discriminant of $C$, and the results of Brumer and Kramer in \cite{BK94} imply that the conductor exponent $v_p(N)$ is bounded by $20,10,9,4$ for $p=2$, $p=3$, $p=5$, $p\ge 7$, respectively.
There are thus only finitely many possibilities for $N$ that can arise for any given $A$, and we can efficiently enumerate them.  Ruling out all but one possibility suffices to determine the conductor $N$, subject to our assumption that \eqref{eq:func} holds.  Ruling out all $N\le N_{\max}$, where $N_{\max}$ determines the conductors we consider ``small'', suffices to prove that $N > N_{\max}$, under the same assumption.  This is often much easier than determining the exact value of $N$.

Let us now sketch the algorithm.  We are given a genus 3 double cover $C\to E$, an upper bound $N_{\max}$ on the conductor $N$ of the associated Prym $A$, an upper bound $r_{\max}$ on $\rad(N)$ (the product of the prime divisors of $N$), and an upper bound $L_{\max}$ on the number of tests of the functional equation \eqref{eq:func} we are willing to apply.
Our goal is to compute $N$ or conclude that one of these bounds has been violated; the algorithm returns \texttt{failure} in the latter case.  In order to make the algorithm practical, we also return \texttt{failure} in situations where the discriminant of $C$ contains prime factors larger than $N_{\max}$, which may cause us to miss small conductor Pryms in situations where the covering curve $C$ has a large prime of almost good reduction, as well as cases where $\Delta$ is not the minimal discriminant of $C$ (we do attempt to minimize the equation for $C$, but bound the effort).
\begin{enumerate}
\setlength{\itemsep}{4pt}
\item[1.] Compute the discriminant $\Delta$ of $C$ and factor it using trial division over primes $p\le \sqrt{N_{\max}}$. If this does not factor $\Delta$ up to a prime-power cofactor, return \texttt{failure}.
\item[2.] For each prime $p|\Delta$ compute lower and upper bounds on $v_p(N)$, taking into account $v_p(\Delta)$ and any applicable bounds from Table \ref{table:reductions}.
\item[3.] Compute the set $S$ of integers $N\le N_{\max}$ with prime factors $p|\Delta$ satisfying the $v_p(N)$ bounds from Step 2 with $\rad(N)\le r_{\max}$.
\item[4.] For each $N\in S$, compute the number $L_N$ of combinations of $\varepsilon$ and pairs $(a_p,a_{p^2})$ over $p|\Delta$ that are compatible with $v_p(N)$, using the Weil bounds and constraints on degrees of bad Euler factors.\footnote{For example, \cite[Corollary II.2]{ST} implies that for $p>5$ we have $\deg L_p(T)= 4-v_p(N)$.}  Sort $S$ by $L_N$ and truncate if needed to make $\sum_{N\in S} L_N\le L_{\max}$.  Return \texttt{failure} if $S=\emptyset$.
\item[5.] Compute $a_p(A)=a_p(C)-a_p(E)$ for $p\nmid\Delta$ with $p\le B \coloneqq b \sqrt{N_{\max}}$.\\
Compute $a_{p^2}(A)=a_{p^2}(C)-a_p(C)a_p(E)+p$ for $p\nmid\Delta$ with $p\le \sqrt{B}$.
\item[6.] For each candidate conductor $N$ in the ordered set $S$, test \eqref{eq:func} using each of the possibilities for $\varepsilon=\pm 1$ and the pairs $(a_p,a_{p^2})$ over $p|\Delta$ determined in Step 4. If the test passes for more than one choice, return \texttt{error}; if the test passes for exactly one choice, return $N$ and the values of $\varepsilon$ and the $(a_p,a_{p^2})$ pairs that worked; otherwise, proceed to the next $N\in S$.
\item[7.] Return \texttt{failure}.
\end{enumerate}

In our computations we used $r_{\max}\coloneqq 2^{10}$ and $N_{\max} \coloneqq 2^{20}$, yielding $B\approx 12288$, and found that $b=12$ was sufficient to ensure the algorithm never returned \texttt{error}.

For the computation in Step 5, for primes $p \le \sqrt{B} \approx 111$ we used na\"ive point-counting to compute $\#E(\F_p)$, $\#C(\F_p)$, and $\#C(\F_{p^2})$, from which we can derive $a_p(A)$ and $a_{p^2}(A)$.  For the remaining $p \le B \approx 12288$ we compute $a_p(A) = a_p(C)-a_p(E)$ by applying average polynomial-time algorithms to compute $a_p(C) \bmod p$ for $p\le B$ in $B(\log B)^{3+o(1)}$ time (we used \texttt{smalljac} \cite{KS} to compute $a_p(E)$).  When $C$ is a hyperelliptic curve $y^2=f(x)$, we use the algorithm described in \cite{HS16,Sut20}, and when $C$ is a smooth plane quartic $f(x,y,z)=0$, we use the algorithm described in \cite{CHS}.  In the latter case we may exploit the fact that $C$ is a double cover of a genus 1 curve via \cite[Remark 5.26]{CHS}.
Note that for $p>64$ the Weil bounds imply $|a_p(A)|\le 4\sqrt{p} \le p/2$, so $a_p(A)$ is uniquely determined by $a_p(A)\bmod p$.

\begin{remark}
The asymptotic complexity of our implementation of Step 5 is $\tilde O(B^{3/2})$.  This is suboptimal, but for $\sqrt{B}\approx 111$ it is the most efficient approach in practice.  There is an asymptotically more efficient approach that uses a combination of the $L$-polynomial lifting algorithms described in \cite{Shi2,Shi3} to lift the $L$-polynomial of $A$ from $\Z/p\Z$ to $\Z$ in $O((\log p)^{2+o(1)})$ expected time, which yields an $\tilde O(B)$ Las Vegas algorithm to compute the $L$-polynomials $L_p(T)$ of $A$ for primes $p\le B$ of good reduction for $C$ that is much more efficient in practice than using Harvey's general-purpose average polynomial-time algorithm for arithmetic schemes \cite{Harvey}.
While we did not use this algorithm to compute conductors, we did use it to identify Prym varieties $A$ with $\mathrm{End}(A_{\Qbar})=\Z$ via Zywina's algorithm \cite{Zywina}, which requires full $L$-polynomials for more values of $p$ than we needed to compute conductors.
\end{remark}

\begin{remark}
Genus 3 curves $C/\Q$ that are not smooth plane quartics are always hyperelliptic over $\Qbar$ but need not admit a hyperelliptic model $y^2=f(x)$ over $\Q$.  There is an average polynomial-time algorithm that can be applied to such curves \cite{HMS}, but as explained in Remark~\ref{remark:hyperelliptic-twisting}, we can always realise a quadratic twist of our Prym using a hyperelliptic model, so we did not use this algorithm.
\end{remark}

By caching precomputed quadratic residues and roots of cubic polynomials over $\F_{p^2}$ for $p \le \sqrt{B}$ we are able to ensure that the na\"ive point-counting performed in Step 5 takes negligible time (a few milliseconds), and the speed of the average polynomial-time algorithms in \cite{CHS,HS16,Sut20} ensured that the total time for Step 5 was typically under 2 seconds.
In most cases the computation of the conductor was dominated by Step 6, since the number of combinations $L_N$ of $\varepsilon$ and pairs $(a_p,a_{p^2})$ for $p|N$ may be quite large.
Keeping $r_{\max}$ small helps to control this, but to ensure that the computations remain manageable we set $L_{\max} \coloneqq 2^{28}$, which forces the algorithm to terminate Step 6 if the number of tests exceeds $L_{\max}$.

With this constraint in place the average running time of the algorithm was around 15 seconds on the 2.45 GHz AMD Milan CPUs that we used.
In roughly 5/6 of the cases the $L_{\max}$ constraint prevented the algorithm from considering all possible $N\le N_{\max}$ with $\rad(N)\le r_{\max}$, so it is very likely that we missed many Pryms whose conductors satisfy our search parameters.  But adding the $L_{\max}$ constraint reduces the average running time of the algorithm by more than a factor of 10, allowing us to process many more Pryms.

\section{Constructing a database of Pryms of small conductor}\label{sect:database}

To construct genus 3 covers $\phi\colon C\to E$ that will produce a Prym variety $A$ of small conductor, we control the primes of bad reduction for both $E$ and $C$ via our selection of $E$, the ramification points of $\phi$, and the number field $K$ (of degree at most 4) over which these ramification points are defined.
We start by placing constraints on the bad primes we will allow.
In order to make the conductor computations manageable, we want the primes dividing the conductor to be fairly small (the larger the prime, the greater the number of possible Euler factors), and we also want the number of bad primes to be small.
We also want to constrain the archimedean valuation of the conductor, but this is actually less critical than controlling the number of possible bad Euler factors, which can lead to a combinatorial explosion in our heuristic algorithm for computing the conductor, as explained in the previous section.

For our database, which we do not claim to be complete in any way, we decided to impose a bound of $r_{\max}\coloneqq 2^{10}$ on the conductor radical $\rad(N)$ (the product of its prime divisors), which ensures that both the size and number of primes dividing the conductor are small.
Subject to this constraint, our goal is to search for Prym varieties with conductors $N\le  N_{\max} \coloneqq 2^{20}$, a bound we chose because it suffices to cover all abelian surfaces with good reduction away from $2$ (this choice turned out to be fortuitous, as explained in Section~\ref{sect:postscript}), but the bound on $N$ plays no role in the initial stage.

We begin by assembling a list of elliptic curves $E/\Q$ whose conductors have radical bounded by $2^{10}$ (this isn't strictly necessary, as primes dividing the conductor of $E$ need not divide the conductor of $A$, but they usually will).
Fortunately the LMFDB contains a large (but far from complete) set of such elliptic curves, including all $E/\Q$ with good reduction away from $2,3,5,7$ and all $E/\Q$ with conductor less than 500\,000.
For the number field $K$, we consider all number fields of degree up to 4 with $|\Delta_K|\le 2^{20}$, all of which are included in the LMFDB.

We then iterate over pairs $(E,K)$ and for each pair we compute the Mordell-Weil group $E(K)$, which we require to be strictly larger than $E(\Q)$ whenever $K\ne \Q$ (otherwise there is no point in considering $K$).
For each pair we provisionally compute a set of generators for the Mordell-Weil groups $E(\Q)$ and $E(K)$ and then iterate over integer linear combinations of these generators to construct ramification points; this set of generators might not be complete because we impose a height bound to keep the computation feasible (indeed, no rigorous algorithm is known to compute $E(\Q)$, or even its rank).
We want to keep the (na\"ive) heights of the points we choose small, so we restrict the absolute values of the integer coefficients we use in our integer linear combinations to at most $5$ and reduce this bound as needed to keep the total number of integer linear combinations reasonably small (we imposed a bound of $3^9=19683$ on the total number of choices of ramification points to consider for each pair).

We then apply the following constraints.
\begin{itemize}
\setlength{\itemsep}{4pt}
\item The ramification points $P_1,P_2,P_3,P_4$ must be distinct, and their sum must be divisible by $2$; in other words, $P_1+P_2+P_3+P_4=2Q$ for some $Q\in E(\Q)$. We now take $D' = [0_E] + [Q]$.
\item The product of the primes $p$ that divide $N_E$ or $\Delta_K$, or the norm of a prime of $K$ modulo which $P_1,P_2,P_3,P_4$ are not distinct must be at most $2^{10}$, as these primes will likely become primes of bad reduction for $C$.
\item The defining polynomial of the quadratic extension of the function field of $E$ obtained by adjoining the square root of the selected function $f$ with $\mathrm{div}(f)=[P_1]+[P_2]+[P_3]+[P_4]-2[Q]-2[0_E]$ must not be too large (we restricted its print length to 2000 characters).
\item The (not necessarily minimal) discriminant $\Delta(C)$ of the genus 3 cover $C/\Q$ can have at most one prime factor greater than $2^{10}$, with $|\Delta|\le 10^{5000}$.
\item The minimal discriminant $\Delta_{\min}(C)$ can contain at most 7 prime divisors, and the product of the primes $p$ with $v_p(\Delta)\in [1,11]$ cannot exceed $2^{10}$.
\end{itemize}

Note that we allow $\Delta_{\min}(C)$ to have radical greater than $2^{10}$ provided that the extra primes occur with exponent at least 12.
The motivation for this is that the $C$ we construct will often have one or more primes of almost good reduction (see Definition~\ref{def:almost-good}) that are bad for $C$ but good for $\Jac(C)$; such primes will not divide the conductor of $A$.
One does not encounter primes of almost good reduction in the current database of genus 2 curves in the LMFDB, but that is only because the current database was constructed using a bound of $10^6$ on the minimal discriminant.
This makes it nearly impossible for almost good primes to appear, since such primes must have discriminant exponent at least 12.
But in the new database of more than 6 million genus 2 curves of small conductor (but not necessarily small discriminant) that is now available in the alpha version of the LMFDB \cite{alpha}, primes of almost good reduction are quite common and appear in almost half the curves.

\begin{remark}\label{rmk:2-torsion}
In addition to enumerating ramification points, we also iterate over non-trivial points in $E(\Q)[2]$, which we may add to the point $Q$ that is used to construct the genus 3 cover.
If $T$ is a 2-torsion point, then $Q+T$ also satisfies all the requirements listed above, but it may produce a different cover and a non-isogenous Prym; see also \Cref{rmk:non-unique}.
\end{remark}

The first phase of the computation iterates over pairs $(E,K)$ and integer linear combinations of $K$-rational points to find genus 3 covers $C\to E$ that satisfy the constraints above, and computes an isogeny invariant of each Prym using the \emph{trace hash} defined in \cite[\S 4.3]{BSSVY}. This is a linear combination of $a_p(A)$ values for $p\in [2^{12},2^{13}]$ computed in the finite field with $2^{61}-1$ elements using coefficients derived from the digits of $\pi$.
This range of $p$ conveniently avoids the primes of bad reduction, and because the trace hash is a linear function, we can compute the trace hash for $A$ as the difference of the trace hashes for $C$ and $E$. This typically takes less than a second to compute using average polynomial-time algorithms.
The trace hash is necessarily an isogeny invariant, and as long as the database is not too large and not too many of the $a_p(A)$ are zero, one expects the trace hash to uniquely distinguish the Prym's isogeny class from other isogeny classes in the dataset.

Of course one always needs to consider the birthday paradox, and some hash collisions may have occurred, but Pryms with distinct trace hashes cannot be isogenous, and this is useful to know.
In our computations we also introduced a \emph{twist hash}, which is computed in the same way as the trace hash, but using $|a_p(A)|$ rather than $a_p(A)$, making it invariant under quadratic twisting as well as isogeny. This twist hash gives an efficient way to certify that two abelian surfaces are not isogenous to any quadratic twists of each other.

In the second phase of the computation we attempt to compute the conductors of all the Pryms $C\to E$ that have been enumerated.
There is a tradeoff between spending effort on computing conductors versus enumerating more Pryms whose conductors might be easier to compute, so we set an upper bound on the total number of conductor and Euler factor combinations we are willing to consider for a single Prym.
Setting this bound to $L_{\max}\coloneqq 2^{28}$ meant that on average we spent about 15 seconds attempting to compute the conductor of each Prym.
As can be seen in the next section, we successfully computed a heuristic conductor below the $2^{20}$ bound in only about 1 in 200 cases, but in nearly 1 in 6 cases we were able to heuristicall conclude that the conductor is greater than $2^{20}$ using fewer than $L_{\max}\coloneqq 2^{28}$ combinations (and around 15 seconds of wall time, on average).

We should emphasise that our conductor computations are heuristic for two distinct reasons: (a) they depend on the unproven assumption that $L(A,s)$ has an analytic continuation that satisfies the functional equation \eqref{eq:func} and, (b) in our computations we only consider conductors that satisfy the bounds imposed by $N\le N_{\max}$ and $\rad(N)\le r_{\max}$, we terminate our computation as soon as we find a conductor and a choice of root number and bad Euler factors that works, and we use floating-point calculations rather than interval arithmetic.  As noted in Section~\ref{sect:conductor}, one can remove heuristic (b) by running a more extensive computation that considers every conductor $N$ supported on primes dividing the minimal discriminant of $C$, uses interval arithmetic to make the numerical tests of \eqref{eq:func} fully rigorous, and verifies that there is only one value of $N$ for which the numerical test succeeds. We have not yet attempted this computation, but we plan to do so in the future.

Notwithstanding these provisions, as can be seen in the next section, for about 80 percent of the nearly half a million Pryms for which we provisionally computed the conductor there was a previously known genus 2 curve with a matching trace hash and the same conductor.  The conductors of these genus 2 curves have all been rigorously computed under assumption (a), and for about 10 percent of them (those for which the modularity criteria of \cite{BCGP} apply), (a) is actually known to hold.  This gives us confidence that our heuristic conductor computations are correct.

\section{Computational results}\label{sect:results}

We undertook the enumeration described in the preceding section in a large parallel computation run on Google's cloud platform, using roughly 250\,000 cores, most of which were 2.45 GHz AMD Milan CPUs (Zen 3 architecture).
The total CPU time spent was on the order of 100 CPU years, with about half spent enumerating Pryms and half spent computing conductors (we chose parameters to ensure a close to 50/50 division of labour).  \Cref{table:counts,table:isogenyclasses} summarise the results of the enumeration phase, while \Cref{table:smallconductor} provides statistics on the Pryms of small conductor that were found, a surprising number of which are isogenous to Jacobians of genus 2 curves over $\Q$, more than 10\,000 of which were not previously known.  See \Cref{sect:postscript} for some particularly interesting examples.

\begin{center}
\begin{table}[htb!]
\begin{tabular}{lr}
\textbf{objects} & \textbf{count}\\\toprule
Elliptic curves $E$ with $\rad(N_E)\le 2^{10}$ in the LMFDB & 315\,607\\
Number fields $K$ of degree $\le 4$ with $|\Delta_K|\le 2^{20}$ and $\rad(\Delta_K)\le 2^{10}$ & 12\,821\\
Pairs $(E,K)$ with $\rad(N_E\Delta_K)\le 2^{10}$ and $n(E_K)>n(E)$ & 42\,504\,851\\
Genus 3 double covers $C\to E$ arising from some pair $(E,K)$ & 117\,615\,804\\ 
\hspace{20pt}that are hyperelliptic & 36\,014\,482\\ 
\hspace{40pt}that required twisting (see \Cref{remark:hyperelliptic-twisting}) & 546\,281\\ 
\hspace{20pt}that are nonhyperelliptic & 81\,601\,322\\ 
\bottomrule
\end{tabular}
\smallskip

\caption{Pryms constructed.  Here $n(E)$ and $n(E_K)$ count Mordell-Weil generators (including torsion generators).  The $E/\Q$ we considered include all $E/\Q$ with $N_E\le 500\,000$ or $\rad(N_E)|210$.}\label{table:counts}
\end{table}
\end{center}
\vspace{-30pt}

\begin{remark}\label{remark:hyperelliptic-twisting}
Every irreducible smooth projective genus 3 curve $C/\Q$ is either a smooth plane quartic or hyperelliptic in the sense that $C$ admits a degree-2 map to a genus zero curve $Y/\Q$.  If $Y\simeq \P^1$ then $C$ admits a rational hyperelliptic model of the form $y^2=f(x)$ with $f\in \Q[x]$, but otherwise $C$ does not admit such a model over $\Q$; it can be defined as the intersection of a hypersurface $w^2=f(x,y,z)$ and a conic $Y\colon g(x,y,z)=0$ in weighted projective space, where $f$ and $g$ are ternary quartic and quadratic forms, respectively.  The base change of $C$ to a quadratic extension $K/\Q$ over which $Y$ has a rational point will admit a hyperelliptic model with coefficients in $K$ whose Shioda invariants (see \cite{Shioda}) are rational.  In general, there may not be a curve defined over $\Q$ with a given set of rational Shioda invariants (the field of definition may differ from the field of moduli, even for hyperelliptic curves with extra automorphisms, see \cite[p.\ 86]{Huggins}), but in our situation we know by construction that $\Q$ is a field of definition for $C$, and that the reduced automorphism group has even order (due to the nonhyperelliptic involution induced by the map $C\to E$), so we can use the algorithm of Lercier and Ritzenthaler \cite{LR12,LRS13} implemented in Magma \cite{magma} to construct a hyperelliptic curve $\tilde C\colon y^2=f(x)$ with $f\in \Q[x]$ that has the same Shioda invariants as $C$; see \cite[Lemma 1.6]{LRS13}.  The curve $\tilde C$ won't necessarily be a degree-2 cover of $E$, but it will be a degree-2 cover of a genus 1 curve $\tilde E$ whose Jacobian is a quadratic twist of $E$, which we can find by enumerating quadratic characters ramified only at the primes of bad reduction of $\tilde C$.  By applying this quadratic twist to the double cover $\tilde C\to \tilde E$ we obtain a double cover $\tilde C'\to E$ whose associated Prym variety is a quadratic twist of our original Prym (even though $\tilde C'$ is \emph{not} a quadratic twist of our original geometrically hyperelliptic $C/\Q$).
\end{remark}

\begin{remark}\label{remark:g2cover}
As explained in \cite{LLRS23}, over an algebraically closed field, we can always find a Jacobian that is isogenous to our Prym.  When $C$ is hyperelliptic this is also true over $\Q$: we have $\Z/2\Z\times\Z/2\Z \hookrightarrow \mathrm{Aut}(C)$, and \cite[Theorem B]{KR} implies that $\Jac(C)\sim \Jac(C_1)\times \Jac(C_2)$ for some curves $C_1,C_2$ of genus $1,2$ (apply Theorem B with $G=\Z/2\times\Z/2$ and $H_1,H_2,H_3\simeq \Z/2$ with $H_3$ generated by the hyperelliptic involution whose corresponding quotient has genus 0).
When $C$ is nonhyperelliptic we have $\Z/2\Z \hookrightarrow \mathrm{Aut}(C)$, which allows us to define $C$ using a model of the form $y^4 + h(x,z)y^2 + f(x,z) = 0$, where $f$ and $h$ are binary forms of degrees 4 and 2, respectively, by choosing a model for which the $\Z/2\Z$ involution is $y\mapsto -y$ (efficient code for doing this can be found in our GitHub repository \cite{github}).  Over $\Qbar$ we can factor $f(x,z)$ into a product of binary quadratic forms
and apply the formulas of Ritzenthaler--Romagny \cite{RR18} to obtain equations for curves $C_1,C_2$ of genus $1,2$ such that $\Jac(C)\sim \Jac(C_1)\times\Jac(C_2)$, with $\Jac(C_2)$ isogenous to our Prym.  Such a factorisation need not exist over $\Q$ (as noted in \cite{RR18}, there is a necessary rationality condition that is not always satisfied), but whenever it does we immediately obtain an equation for $C_2$.  Magma code that efficiently computes a genus 2 curve $C_2$ with $\Jac(C_2)\sim A$ in both the hyperelliptic and nonhyperelliptic cases can be found in our GitHub repository \cite{github}.  We exploit this whenever possible, as it is much easier to compute the conductor (and the $L$-function) of $\Jac(C_2)$ than of the Prym $A$ which we can access only indirectly.
\end{remark}

\begin{remark}\label{remark:g1cover}
As noted in \Cref{remark:g2cover}, when the genus 3 curve $C$ admits both the Prym map $C\to E$ and a degree-2 map $C\to C_2$ to a genus 2 curve, then we have associated isogeny decompositions $\Jac(C)\sim A\times E$ and $\Jac(C)\sim \Jac(C_2)\times E'$, where $A$ is the abelian surface corresponding to the Prym variety and $E'$ is an elliptic curve which may be isogenous to $E$ but need not be.
If $E'\sim E$ then $A\sim\Jac(C_2)$ and we have $L(A,s)=L(C_2,s)$. This will necessarily occur when there is only one involution to work with, as the map $C \to C_1$ induced by our even model for $C$ (see \Cref{remark:g2cover}) will force $\Jac(C_1)\sim E$, but this need not hold when there are additional involutions, for example, when $C$ is a Ciani quartic with automorphism group $\Z/2\Z\times \Z/2\Z$.  This is actually useful, because if we find that $\Jac(C_1)$ is not isogenous to $E$ then we immediately get a factorisation of $\Jac(C)$ and our Prym into a product of elliptic curves, that is, there must exist a third elliptic curve $E''$ for which $\Jac(C)\sim E\times E'\times E''$. Indeed, $A\sim E'\times E''$ and $\Jac(C_2)\sim E\times E''$. 
With an equation for $C_2$ in hand, it is not hard to find $E''$.
\end{remark}

Data for the small conductor Pryms whose isogeny classes are tabulated in Table~\ref{table:smallconductor} is available in the GitHub repository associated to this paper \cite{github}.  For each genus 3 double cover $C\to E$ we provide the following data:
\begin{itemize}
\setlength{\itemsep}{4pt}
    \item The conductor $N$ (subject to heuristics (a) and (b) noted in Section~\ref{sect:database}).
    \item The trace hash $h$ (an isogeny class invariant).
    \item An integral equation for $C$ (either a smooth plane quartic $f(x,y,z)=0$ or a hyperelliptic curve $y^2 +h(x)y=f(x)$).
    \item Weierstrass coefficients $a_1,a_2,a_3,a_4,a_6$ for a minimal integral model of $E$.
    \item The root number $\varepsilon$ and the Euler factors $L_p(T)$ for primes $p|\Delta(C)$ that were used to deduce $N$ via the functional equation \eqref{eq:func}.
    \item In cases where \Cref{remark:g2cover} applies we provide an integral equation for the genus 2 curve $C_2$ whose Jacobian is isogenous to the Prym variety $A$.
\end{itemize}
The previously unknown isogeny classes of genus 2 curves over $\Q$ of conductor $N\le 2^{20}$ may contain multiple $\Q$-isomorphism classes of Jacobians.  These can be exhaustively enumerated via \cite{vBCCK}, and we plan to do this before adding our data to the LMFDB.

\begin{center}
\begin{table}[htp!]
\begin{tabular}{lr}
\textbf{objects} & \textbf{lower bound}\\\toprule
Isogeny classes of Pryms & 80\,904\,678\\ 
\hspace{20pt}that contain a genus 2 Jacobian via \Cref{remark:g2cover} & 32\,754\,084\\ 
\hspace{40pt}that are hyperelliptic& 27\,390\,870\\ 
\hspace{20pt}that are nonhyperelliptic & 55\,699\,050\\ 
\hspace{40pt}that contain a genus 2 Jacobian via \Cref{remark:g2cover} & 5\,833\,651\\ 
\midrule
Isogeny classes of Pryms up to quadratic twist & 25\,549\,552\\ 
\hspace{20pt}that contain a genus 2 Jacobian via \Cref{remark:g2cover} & 10\,521\,522\\ 
\hspace{40pt}that are hyperelliptic & 7\,261\,867\\ 
\hspace{20pt}that are nonhyperelliptic & 20\,061\,326\\ 
\hspace{40pt}that contain a genus 2 Jacobian via \Cref{remark:g2cover} & 3\,912\,828\\ 
\bottomrule
\end{tabular}
\smallskip

\caption{Isogeny classes of Pryms enumerated.}
\label{table:isogenyclasses}
\end{table}
\end{center}

\begin{center}
\begin{table}[htp!]
\begin{tabular}{lr}
\textbf{objects} & \textbf{count}\\\toprule
Isogeny classes of Pryms of conductor $\le 2^{20}$ & 457\,885\\ 
\hspace{20pt} that contain a previously known genus 2 Jacobian & 376\,386\\ 
\hspace{20pt} that contain a genus 2 Jacobian via \Cref{remark:g2cover} & 303\,567\\ 
\hspace{40pt} that were previously unknown & 10\,471\\ 
Isogeny classes of generic ($\End(A_{\Qbar})=\Z$) Pryms of conductor $\le 2^{20}$ & 236\,850\\ 
\hspace{20pt} that contain a previously known genus 2 Jacobian & 217\,598\\ 
\hspace{20pt} that contain a genus 2 Jacobian via \Cref{remark:g2cover} & 210\,689\\ 
\hspace{40pt} that were previously unknown & 8\,364\\ 
\bottomrule
\end{tabular}
\smallskip

\caption{Isogeny classes of Pryms with conductor $\le 2^{20}$; all counts assume that heuristically computed conductors are correct and isogeny classes are uniquely identified by their trace hashes.}
\label{table:smallconductor}
\end{table}
\end{center}

\subsection{\texorpdfstring{Genus 2 curves over $\Q$ of 2-power conductor}{Genus 2 curves over the rationals of 2-power conductor}}\label{sect:postscript}
In his 1996 invited talk \textit{Computational aspects of curves of genus at least $2$} at the second Algorithmic Number Theory Symposium (ANTS II), Bjorn Poonen \cite{Poonen} proposed the problem of enumerating all genus 2 curves over $\Q$ whose Jacobians have good reduction away from 2; this was shortly after the easier problem of enumerating all genus 2 curves over $\Q$ with good reduction away from 2 had been solved by Smart \cite{Smart}, building on previous joint work with Merriman \cite{MS93}.  Little progress was made on this problem in the next two decades, but there is substantial recent progress due to Robin Visser.  In his doctoral thesis \cite{Visser} Visser presents a list of 512 genus 2 curves $C/\Q$ whose Jacobians have good reduction away from 2, obtained through a wide variety of means and extensive search.  We are happy to add ten new curves to this list:
\begin{center}
\footnotesize
\begin{tabular}{lll}
$N$ & $|D|$ & curves\\\midrule
$2^{14}$ & $2^{54}3^{22}11^{22}$ & $\pm y^2 = 7161x^6+7722x^5-40887x^4-74844x^3+46431x^2+17226x-5313$\\
$2^{16}$ & $2^{49}3^{22}11^{22}$ & $\pm y^2 = 5115x^6+26928x^5+18117x^4+5016x^3+19305x^2-22968x+7359$\\
$2^{20}$ & $2^{54}13^{22}$ & $\pm y^2 = 3107x^6+17576x^5+18642x^4+37284x^2-70304x+24856$\\
$2^{20}$ & $2^{54}13^{22}$ & $\pm y^2 = 1287x^6-9854x^5-5993x^4+100308x^3+93873x^2-9854x-18863$\\
$2^{20}$ & $2^{54}13^{22}$ & $\pm y^2 = 7501x^6+27430x^5+49933x^4+75452x^3+37947x^2+27430x-25077$\\\bottomrule
\end{tabular}
\end{center}
\bigskip
All ten of these curves have primes of almost good reduction that divide the minimal discriminant of the curve but not the conductor of the Jacobian.  These are all (non-quadratic) twists of curves on Visser's list and have Jacobians isogenous to Jacobians of curves on Visser's list, but neither the twists nor the isogenies are easy to discover computationally. We found them simply because their Jacobians happen to be isogenous to Pryms found by our search.  In total we found 290 isomorphism classes of genus 2 curves $C/\Q$ whose Jacobians have $2$-power conductor; the other 280 already appear in Visser's list.

\section*{Competing Interests}
The authors have no competing interests to declare that are relevant
to the content of this article.

\section*{Data availability}
All generated data relating to this article is accessible at \url{https://github.com/AndrewVSutherland/Genus3Covers}.

\end{document}